\documentclass[12pt,leqno,twoside]{amsart}
\usepackage{amsmath, amscd, amssymb, amsfonts, tikz-cd, enumitem, appendix} %, biblatex}
\usepackage{latexsym}
\usepackage[all,cmtip]{xy}
\usepackage{graphicx}
\usepackage[colorlinks=true, pdfstartview=FitV, linkcolor=blue, citecolor=blue, urlcolor=blue]{hyperref}
\usepackage{subcaption} 
\usepackage{mathtools, dsfont}

\usepackage{helvet}
\usepackage[margin=1in]{geometry}

\theoremstyle{definition}
\newtheorem{thm}{Theorem}[section]
\newtheorem{cor}[thm]{Corollary}
\newtheorem{defn}[thm]{Definition}
\newtheorem{example}[thm]{Example}

\newtheorem{lemma}[thm]{Lemma}
\newtheorem{prop}[thm]{Proposition}
\newtheorem{remark}[thm]{Remark}
\newtheorem{conj}[thm]{Conjecture}

\newtheorem*{assumption}{Assumption}

\numberwithin{equation}{section}

\newcommand{\Z}{\mathbb{Z}}

\newcommand{\Q}{\mathbb{Q}}

\newcommand{\C}{\mathbb{C}}
\newcommand{\K}{ A}

\renewcommand{\P}{\mathbb{P}}

\newcommand{\bF}{\mathbb{F}}

\newcommand{\F}{\mathbb{F}}

\def\part{\partial}

\def\ker{\mathrm{Ker}}

\def\bd{\begin{defn}}
\def\ed{\end{defn}}
\def\bt{\begin{thm}}
\def\et{\end{thm}}
\def\br{\begin{remark}}
\def\er{\end{remark}}
\def\bc{\begin{cor}}
\def\ec{\end{cor}}
\def\bp{\begin{prop}}
\def\ep{\end{prop}}
\def\be{\begin{equation}}
\def\ee{\end{equation}}
\def\bn{\begin{enumerate}}
\def\en{\end{enumerate}}
\def\ba{\begin{array}}
\def\ea{\end{array}}
\def\bex{\begin{example}}
\def\eex{\end{example}}

\newcommand\sA{{\mathcal A}}

\newcommand\sH{{\mathcal H}}

\newcommand\sP{{\mathcal P}}

\newcommand\sF{{\mathcal F}}

\newcommand\sG{{\mathcal G}}

\newcommand\sL{\mathcal{L}}
\newcommand\sB{\mathcal{B}}

\newcommand\sM{{\mathfrak M}}

\DeclareMathOperator{\Exp}{Exp}

\def\bC{\mathbb{C}}

\title[Non-abelian Mellin Transformations]{Non-abelian Mellin Transformations and Applications}

\begin{document}

\author[Y. Liu ]{Yongqiang Liu}
\address{Y. Liu : The Institute of Geometry and Physics, University of Science and Technology of China, 96 Jinzhai Road, Hefei 230026 P.R. China}
\email{liuyq@ustc.edu.cn}

\author[L. Maxim ]{Lauren\c{t}iu Maxim}
\address{L. Maxim : Department of Mathematics, University of Wisconsin-Madison, 480 Lincoln Drive, Madison WI 53706-1388, USA}
\email {maxim@math.wisc.edu}

\author[B. Wang]{Botong Wang}
\address{B. Wang : Department of Mathematics, University of Wisconsin-Madison, 480 Lincoln Drive, Madison WI 53706-1388, USA}
\email {wang@math.wisc.edu}

\date{\today}



\date{\today}

\begin{abstract} 
We study non-abelian versions of the Mellin transformations, originally introduced by Gabber-Loeser on complex affine tori. Our main result is a generalization to the non-abelian context and with arbitrary coefficients of the $t$-exactness of Gabber-Loeser's Mellin transformation. As an intermediate step, we obtain vanishing results for the Sabbah specialization functors.  %\textcolor{blue}{ show  the $t$-exactness of the multi-variable Sabbah specialization functor, generalizing the well known result about the $t$-exactness of the nearby cycle functor.} 
 Our main application is to construct new examples of duality spaces in the sense of Bieri-Eckmann, generalizing results of Denham-Suciu.
 % show that linear, toric and elliptic arrangement complements are duality spaces in the sense of Bieri-Eckmann, a result initially proved by Denham-Suciu.  We also give new examples of duality spaces that are non-affine or singular varieties.
\end{abstract}

\maketitle

\section{Introduction} 

The geometry and topology of a complex algebraic or analytic variety can be studied via the cohomology groups of its coherent and constructible sheaves. The Fourier-Mukai transformation on the coherent side, and the Mellin transformation on the constructible side, are functors which allow one to compute the cohomology groups of a coherent or constructible sheaf twisted by all (topologically trivial) line bundles or local systems. As the Fourier-Mukai transformation has become an essential tool in birational geometry, the Mellin transformation has also proved useful in the study of perverse sheaves, especially on complex affine tori \cite{GL, LMWb}, abelian varieties \cite{Sc15, BSS} and, more generally, on semi-abelian varieties \cite{Kra, LMWa, LMWc}. In particular, whether a constructible complex is a perverse sheaf or not can be completely determined by its Mellin transformation. 

In this paper, we establish non-abelian generalizations of the $t$-exactness result of Gabber-Loeser \cite{GL} to certain families of Stein manifolds, e.g., complements of essential hyperplane arrangements. In this general setting, we reduce the global $t$-exactness of the Mellin transformation to a certain local vanishing result for the multi-variable Sabbah specialization functor \cite{Sa}. 
In fact we show  the $t$-exactness of the multi-variable Sabbah specialization functor, generalizing the well known result about the $t$-exactness of the nearby cycle functor. 
As a special case of this local vanishing result, we prove a local version of the $t$-exactness result of Gabber-Loeser. Using the $t$-exactness of the non-abelian Mellin transformations, we construct new families of duality spaces which generalize the ones of Denham-Suciu \cite{DS}.

\medskip

Let $\K$ be a Noetherian commutative ring of finite cohomological dimension. Let $U$ be a complex manifold with fundamental group $G$. Let $\sL_U$ be the universal $\K[G]$-local system on $U$. Denote by $q:U\to \textrm{pt}$ the projection to a point space. For an $\K$-constructible complex $\sF$ on $U$, we define its {\it Mellin transformations} by 
\[
\sM^U_*(\sF)\coloneqq Rq_*(\sF\otimes_\K \sL_U)\quad\text{and}\quad \sM^U_!(\sF)\coloneqq Rq_!(\sF\otimes_\K \sL_U).
\]
We omit the upperscript $U$ when there is no risk of confusion. These are non-abelian counterparts of similar transformations introduced by Gabber-Loeser \cite{GL} on complex affine tori. 

Our main result is a generalization to the non-abelian context of Gabber-Loeser's t-exactness of the Mellin transformation $\sM_*$ (see \cite[Theorem 3.4.1]{GL} and  \cite[Theorem 3.2]{LMWb}). 
\begin{thm}\label{thm_main}
Let $U$ be a complex manifold with a smooth compactification $U\subset X$, such that the boundary divisor $E=\bigcup_{1\leq k\leq n}E_k$ is a simple normal crossing divisor. Assume that the following properties hold.
\begin{enumerate}
\item For any subset $I\subset \{1, \ldots, n\}$, $E_I^\circ:=\bigcap_{k\in I}E_k\setminus \bigcup_{l\notin I}E_l$ is either empty or a Stein manifold. When $I=\emptyset$, this means that $U=X\setminus E$ is a Stein manifold. 
\item For any point $x\in E$, the local fundamental group of $U$ at $x$ maps injectively into the fundamental group of $U$, i.e., the homomorphism $\pi_1(U_x)\to \pi_1(U)$ induced by inclusion is injective, where $U_x=B_x\cap U$ with $B_x$ a small enough complex ball in $X$ centered at $x$. \end{enumerate}
Then for any $\K$-perverse sheaf $\sP$ on $X$, the Mellin transformation $\sM_*(\sP|_U)$ is concentrated in degree zero. In other words, the functor
\[
\sF\mapsto \sM_*(\sF|_U): D^b_c(X, \K)\to D^b(\K[G])
\]
is  t-exact  with respect to the perverse t-structure on $D^b_c(X, \K)$ and the standard t-structure on $D^b(\K[G])$. 
\end{thm}

\br\label{remark_stratification}
If $U$ is algebraic, the conclusion of the above theorem can be reformulated to the assertion that the functor 
\[ \sM_*: D^b_c(U, \K)\to D^b(\K[G]) \]
is t-exact. Indeed, after choosing an algebraic compactification $j:U \hookrightarrow X$ satisfying the properties of the theorem, then a constructible complex $\sF$ in $U$ is the restriction of the constructible complex $Rj_*\sF$ on $X$. A similar statement holds in the analytic category, provided that one works with a fixed Whitney stratification $\mathcal{S}$ of the pair $(X,E)$ and constructibility is taken with respect to $\mathcal{S}$ (e.g., see \cite[Theorem 2.6(c)]{MS}). 
\er

Examples of varieties $U\subset X$ satisfying the above conditions include: complements of essential hyperplane arrangements, toric arrangements and elliptic arrangements in their respective wonderful compactifications, as well as complements of at least $n+1$ general hyperplane sections in a projective manifold of dimension $n$ (see \cite[Section 2.3]{DS}). 

Notice that the standard inclusion $(\bC^*)^n\hookrightarrow \P^n$ satisfies both conditions in Theorem \ref{thm_main}. Thus, we get the following generalization of Gabber-Loeser's $t$-exactness theorem to arbitrary coefficients, where constructibility is taken in the algebraic sense. We also remark that the original proof of Gabber-Loeser does not apply to this general setting, since it uses in an essential way the artinian and noetherian properties of the category of perverse sheaves with field coefficients.

\begin{cor}
The Mellin transformation
\[
\sM_*: D^b_c((\bC^*)^n, \K)\to D^b(\K[\Z^n])
\]
is $t$-exact with respect to the perverse $t$-structure on $D^b_c((\bC^*)^n, \K)$ and the standard $t$-structure on $D^b(\K[\Z^n])$. 
\end{cor}

As another application of Theorem \ref{thm_main}, we obtain new examples of duality spaces (in the sense of Bieri-Eckmann \cite{BE}), that are non-affine or singular varieties. In particular, we recast the fact, initially proved by Denham-Suciu \cite{DS} by different methods, that linear, toric and elliptic arrangement complements are duality spaces. % We also give new examples  of duality spaces that are non-affine or singular varieties. 

As a key step in proving Theorem \ref{thm_main}, we obtain local vanishing results about Sabbah's specialization functors (Theorem \ref{thm_local}). As a special case, we obtain a local version of the $t$-exactness result of Gabber-Loeser, where constructibility is taken in the analytic category.

\begin{thm}
Let $\sP$ be an $A$-perverse sheaf defined in a neighborhood of $0\in \C^n$. Let $B\subset \C^n$ be a small ball centered at the origin, and let $B^\circ$ be the complement of all coordinate hyperplanes in $B$. Let $\sL_{B^\circ}$ be the universal $\pi_1(B^\circ)$-local system on $B^\circ$. Then,
\[
H^k\big(B^\circ, \sP|_{B^\circ}\otimes_A \sL_{B^\circ}\big)=0\quad \text{for any}\; k\neq 0.
\]
\end{thm}

In this paper, we make essential use of the language of derived categories and perverse sheaves (see, e.g., \cite{KS}, \cite{Schu}, \cite{Di}, \cite{Ma} and \cite{MS} for comprehensive references). 

\medskip

\noindent\textbf{Acknowledgments.} We thank Nero Budur for helpful conversations. Y. Liu is partially supported by the starting grant KY2340000123 from University of Science and Technology of China, National Natural Science Funds of China (Grant No. 12001511),  the Project of Stable Support for Youth Team in Basic Research Field, CAS (YSBR-001),  the project ``Analysis and Geometry on Bundles" of Ministry of Science and Technology of the People's Republic of China, National Key Research and Development Project SQ2020YFA070080 and  Fundamental Research Funds for the Central Universities. L.  Maxim is  partially  supported  by the Simons Foundation  (Collaboration Grant  \#567077),  by the Romanian Ministry of National Education (CNCS-UEFISCDI grant PN-III-P4-ID-PCE-2020-0029), and by a  Labex CEMPI grant (ANR-11-LABX-0007-01). B. Wang is partially supported a Sloan Fellowship and a WARF research grant.

\section{Preliminaries}
\subsection{Universal local system}
Let $X$ be a connected locally contractible topological space, with base point $x$. Let $G=\pi_1(X, x)$, and let $p: \tilde{X}\to X$ be the universal covering map. Let $$\sL_X:=p_! \K_{\tilde{X}},$$
where $\K_{\tilde{X}}$ is the constant sheaf with stalk $A$ on $\tilde{X}$.
 Then $\sL_X$ is a local system of rank one free right $\K[G]$-modules (see \cite[Remark 2.3]{LMWd} for discussions about the right action). Equivalently, $\sL_X$ can be defined as the rank one $A[G]$-local system such that the stalk at $x$ is equal to $A[G]$ and the monodromy action is defined as the left multiplication of $G$ on $A[G]$. 

\begin{remark}\label{remark_universal}
We call $\sL_X$ the \emph{universal local system} of $X$ for the following reasons. Given any $A$-module representation $\rho: G\to \mathrm{Aut}_\K(V)$, we can regard $V$ as a left $A[G]$-module. Then, we have an $\K$-local system $\sL_X\otimes_{\K[G]}V$, whose monodromy action is precisely $\rho$. Moreover, every $A$-local system on $X$ can be obtained this way. 
\end{remark}

\begin{lemma}\label{lemma_YZ}
Let $(Y, y)$ and $(Z, z)$ be two path-connected locally contractible topological spaces with base points. Let $\sL_Y$ and $\sL_Z$ be the universal $\K[\pi_1(Y, y)]$- and $\K[\pi_1(Z, z)]$-local systems on $Y$ and $Z$, respectively. Let $g: Y\to Z$ be a continuous map with $g(y)=z$. If $g_*: \pi_1(Y, y)\to \pi_1(Z, z)$ is injective, then as an $\K[\pi_1(Y, y)]$-local system, $g^*(\sL_Z)$ is a direct sum of copies of $\sL_Y$ indexed by the right cosets $g_*\pi_1(Y)\backslash \pi_1(Z)$. 
\end{lemma}
\begin{proof} 
By definition, the local system $g^*(\sL_Z)$ has stalk $A[\pi_1(Z, z)]$ at $y$, and the monodromy action of $\alpha\in \pi_1(Y, y)$ is equal to the left multiplication of $g_*{\alpha}$. As a left $A[\pi_1(Y, y)]$-module, $A[\pi_1(Z, z)]$ is free and the summands are parametrized by the right cosets $g_*\pi_1(Y)\backslash \pi_1(Z)$. Thus, the assertion of the lemma follows. 
\end{proof}

\subsection{(Weakly) constructible complexes, perverse sheaves, and Artin's vanishing.}
Recall that a sheaf $\sF$ of $A$-modules on a complex algebraic or analytic variety $X$ is said to be weakly constructible if there is a Whitney stratification $\mathcal{S}$ of $X$ so that the restriction $\sF|_S$ of $\sF$ to every stratum $S \in \mathcal{S}$ is an $A$-local system. 
We say that $\sF$ is constructible if, moreover, the stalks $\sF_x$ for all $x \in X$ are finitely generated $A$-modules. Let $D^b(X, A)$ be the bounded derived category of complexes of sheaves of $A$-modules on $X$. A bounded complex $\sF \in D^b(X,A)$ is called (weakly) constructible if all its cohomology sheaves $\mathcal{H}^j(\sF)$ are (weakly) constructible.
Let $D^b_{(w)c}(X, A)$ be the full triangulated subcategory of $D^b(X,A)$ consisting of (weakly) constructible complexes. 

The category $D^b_{(w)c}(X, A)$ is endowed with the perverse t-structure, i.e., two strictly full subcategories  ${^pD}_{(w)c}^{\leq 0}(X, A)$ and ${^pD}_{(w)c}^{\geq 0}(X, A)$ defined by stalk and, resp., costalk vanishing conditions as follows: if $\sF \in D^b_{(w)c}(X, A)$ is constructible with respect to a Whitney stratification $\mathcal{S}$ and we denote by $i_x:\{x\} \hookrightarrow X$ the point inclusion, then:
\begin{itemize}
\item[(a)] $\sF \in {^pD}_{(w)c}^{\leq 0}(X, A) \iff \forall S \in \mathcal{S}, \forall x\in S: H^k(i_x^*\sF)=0 \  \textit{for  all } \ k>-\dim S$,
\item[(b)] $\sF \in {^pD}_{(w)c}^{\geq 0}(X, A) \iff \forall S \in \mathcal{S}, \forall x\in S: H^k(i_x^!\sF)=0 \  \textit{for  all } \ k<\dim S$.
\end{itemize}
The heart of the perverse t-structure is the category of (weakly) $A$-perverse sheaves on $X$. 

Artin's vanishing theorem for perverse sheaves is a key ingredient of both the local and global vanishing results in this paper. We recall here the version for weakly constructible complexes.

\begin{thm}\label{av}
(\cite[Theorem 10.3.8]{KS}, \cite[Corollary 6.1.2]{Schu}, \cite[Theorem 3.64]{MS}) \newline
Let $X$ be a Stein manifold. 
\begin{enumerate}
\item For any $\sF \in {^pD}_{wc}^{\leq 0}(X, A)$,  $H^k(X, \sF)=0$ for $k>0$.
\item For any $\sF \in {^pD}_{wc}^{\geq 0}(X, A)$, $H^k_c(X, \sF)=0$ for $k<0$.
\end{enumerate}
\end{thm}

\subsection{Sabbah's specialization complex}\label{sec2.3}
In this subsection, let $X$ be a connected complex manifold. For $1\leq k\leq n$, let $f_k\colon X\to \C$ be holomorphic functions, and let $D_k=f_k^{-1}(0)$ be the corresponding divisors. Here, we do not require that each $f_k$ is irreducible, but we assume that different $f_k$'s do not share common {irreducible} factors. Set $\bigcup_{1\leq k\leq n}D_k=D$, with complement $X\setminus D=U$. Let $$F=(f_1, \ldots, f_n)\colon X\to \C^n,$$ and denote by $F_U\colon U\to (\C^*)^n$ the restriction of $F$ to $U$. Let $i: D\hookrightarrow X$ and $j: U\hookrightarrow X$ be the closed and open embeddings, respectively. Let $\sL_{(\C^*)^n}$ be the universal local system on $(\C^*)^n$, and let
$$\sL^F_{U}=F_U^*(\sL_{(\C^*)^n}).$$  %

Under the above notations, we make the following.
\begin{defn}\label{defn_Sabbah}
The \emph{Sabbah's specialization functor} is defined as
\[
\Psi_F: D^b_{(w)c}(X, A)\to D^b_{(w)c}(D, R), \quad \sF\mapsto i^*Rj_*\big(\sF|_U\otimes_A \sL^F_{U}\big)
\]
where $R=A[\pi_1((\C^*)^n)]$.
\end{defn}

\begin{remark}\label{remsab}
\begin{enumerate}
\item The above definition is similar to that of \cite[Definition 3.2]{Bu}, and it differs slightly from Sabbah's initial definition \cite[Def. 2.2.7]{Sa} where one has to restrict further to $\cap_k D_k$. Both \cite{Bu} and \cite{Sa} assume $A=\C$.
\item  When $n=1$ and $A$ is a field, stalkwise $\Psi_F$ is noncanonically isomorphic to Deligne's (shifted/perverse) nearby cycle functor $\psi_F[-1]$. In this case, $\Psi_F$ is an exact functor with respect to the perverse t-structures (see \cite[Theorem 1.2]{Br}). 
\end{enumerate}
\end{remark}

\section{$t$-exactness of the Sabbah specialization functor}\label{t_exact}
As already mentioned in Remark \ref{remsab}, it is known that when defined over a field, the univariate Sabbah specialization functor is exact with respect to the perverse $t$-structures (see \cite[Theorem 1.2]{Br}). The proof uses the stalkwise isomorphism between the univariate Sabbah specialization functor and the perverse nearby cycle functor to conclude the right t-exactness, and then uses Verdier duality to deduce the left t-exactness. However, this proof does not work when the ground field is replaced by a general ring, since the Verdier duality may not exchange the subcategories $^pD^{\leq 0}$ and $^pD^{\geq 0}$ anymore. In this section, we prove the t-exactness of the univariate Sabbah specialization functor over a general ring (Theorem \ref{thm_A}) by showing a costalk formula (Proposition \ref{prop_milnor}). As a consequence, we prove that the Sabbah specialization functor in any number of variables is also t-exact (Corollary \ref{cor_Sabbah}). Throughout this section, we work with bounded weakly constructible complexes.

Let $X$ be a complex manifold, and let $f\colon X\to \C$ be a holomorphic function. Let $X_0=f^{-1}(0)$, and set $U=f^{-1}(\C^*)$. Let $i\colon X_0\hookrightarrow X$ and $j\colon U\hookrightarrow X$ be the closed and open embeddings, respectively. 

Fixing, as before, a commutative Noetherian ring $A$ of finite cohomological dimension, we let $\sL_{\C^*}$ be the universal $A[\pi_1(\C^*)]$-local system on $\C^*$, and denote its pullback to $U$ by $$\sL^f_{U}:=f^* \sL_{\C^*}.$$ 
Given a weakly $A$-constructible complex $\sF$ on $X$, 
the univariate Sabbah specialization complex of $\sF$ is the following object in $D^b_{wc}(X_0, A[\pi_1(\C^*)])$:
\[
\Psi_f(\sF)=i^{*}Rj_*\big(\sF|_{U}\otimes_A \sL^f_{U}\big).
\]

%We prove the following.
\begin{thm}\label{thm_A}
The Sabbah specialization functor 
\[
\Psi_f: D^b_{wc}(X, A)\to D^b_{wc}\big(X_0, A[\pi_1(\C^*)]\big)
\]
is t-exact with respect to the perverse $t$-structures. 
\end{thm}

Fixing a chart of $X$ near $x \in X_0$, we consider two real valued functions on this chart: $r$ is the Euclidean distance to $x$, and $d$ is the function given by $d(y)=|f(y)-f(x)|$. 
\begin{lemma}
Let $\sF$ be a weakly $A$-constructible complex on $X$. Choose $0< \epsilon << \delta << 1$. Define 
\[
\Pi_{\epsilon, \delta}\coloneqq \{y\in X\mid r(y)<\delta, 0<d(y)<\epsilon\}\quad \text{and}\quad \Delta^*_\epsilon\coloneqq \{z\in \C^*\mid |z|<\epsilon\}.
\]
Let $f': \Pi_{\epsilon, \delta}\to \Delta^*_\epsilon$ be the restriction of $f$. Then $Rf'_!(\sF|_{\Pi_{\epsilon, \delta}})$ is (cohomologically) locally constant on $\Delta^*_\epsilon$. 
\end{lemma}
\begin{proof}
It is a well-known fact that (see, e.g.,  \cite[Corollary 4.2.2]{Schu})  
\[
Rf'_!(\sF|_{\Pi_{\epsilon, \delta}})\cong \mathbb{D}\, Rf'_*\, \mathbb{D}(\sF|_{\Pi_{\epsilon, \delta}}),
\]
where $\mathbb{D}$ denotes the Verdier dualizing functor. By \cite[Definition 5.1.1 and Example 5.1.4]{Schu}, $Rf'_*\, \mathbb{D}(\sF|_{\Pi_{\epsilon, \delta}})$ is locally constant (i.e., it has locally constant cohomology sheaves). Hence, $Rf'_!(\sF|_{\Pi_{\epsilon, \delta}})$ is also locally constant. 
\end{proof}

The following costalk calculation is essential for proving Theorem \ref{thm_A}.
\begin{prop}\label{prop_milnor}
Let $\sF$ be a weakly $A$-constructible complex on $X$. Then,
\begin{equation}\label{eq_A1}
H_x^k\big(X_0, \Psi_f(\sF)\big)=H^{k-1}_c\big(M_f, \sF|_{M_f}\big)
\end{equation}
where $M_f$ is a local Milnor fiber of $f$ at $x \in X_0$ and $H^*_x(-)$ denotes the local cohomology at $x\in X$. 
\end{prop}
\begin{proof}
By applying the attaching triangle $$j_!j^! \to id \to i_*i^* \xrightarrow{+1}$$ to 
$Rj_*(\sF|_U\otimes \sL^f_{U})$, we get the distinguished triangle
\begin{equation}\label{atr}
j_!(\sF|_U\otimes \sL^f_{U})\to Rj_*(\sF|_U\otimes \sL^f_{U})\to i_* i^{*} Rj_*(\sF|_U\otimes \sL^f_{U})\xrightarrow{+1}.
\end{equation}

Let $i_x: \{x\}\hookrightarrow X$ and $k_x: \{x\}\hookrightarrow X_0$ be the inclusion maps, and note that $i_x=i \circ k_x$. Applying the  functor $i_x^!$ to \eqref{atr} and using the fact that $i_x^!Rj_*=0$, 
we get from the corresponding cohomology long exact sequence  that, for any $k \in \Z$,
\[
H^k\big(i_x^!i_* i^{*} Rj_*(\sF|_U\otimes \sL^f_{U})\big)\cong H^{k+1}\big(i_x^!j_!(\sF|_U\otimes \sL^f_{U})\big).
\]
Note that 
\[
i_x^!i_* i^{*} Rj_*(\sF|_U\otimes \sL^f_{U}) = k_x^! i^!i_*i^{*} Rj_*(\sF|_U\otimes \sL^f_{U}) =
 k_x^! \Psi_f(\sF).
\]
Hence,  for any $k \in \Z$ we get an isomorphism 
\begin{equation}\label{eq_A2}
H_x^k\big(X_0, \Psi_f(\sF)\big):=H^k( k_x^! \Psi_f(\sF))\cong H^{k+1}\big(i_x^!j_!(\sF|_U\otimes \sL^f_{U})\big).
\end{equation}

For $0<\epsilon<<\delta <<1$, we have
\begin{align}\label{eq_A3}
\begin{split}
H^{k+1}\big(i_x^!j_!(\sF|_U\otimes \sL^f_{U})\big)&\cong H^{k+1}_c\big(\{y\in X\mid r(y)<\delta, d(y)<\epsilon\}, j_!(\sF|_U\otimes \sL^f_{U})\big)\\
&\cong H^{k+1}_c\big(\{y\in X\mid r(y)<\delta, 0<d(y)<\epsilon\}, \sF|_U\otimes \sL^f_{U}\big)\\
&\cong H^{k+1}_c\big(\Delta^\circ_\epsilon, Rf'_!(\sF|_{\Pi_{\epsilon, \delta}})\otimes \sL_{\C^*}|_{\Delta^*_\epsilon}\big)
\end{split}
\end{align}
where the first isomorphism can be deduced, e.g., from \cite[Proposition 7.2.5]{Ma}, and the last follows from the projection formula. 

Let $\Exp: \C\to \C^*$ be the universal covering map, and let $\Exp_\epsilon: \Exp^{-1}(\Delta_\epsilon^\circ)\to \Delta_\epsilon^\circ$ be its restriction. Then $\sL_{\C^*}\cong \Exp_! \underline{A}_{\C}$. Moreover, by the projection formula, we have
\begin{equation}\label{eq_A4}
\begin{split}
H^{k+1}_c\big(\Delta^\circ_\epsilon, Rf'_!(\sF|_{\Pi_{\epsilon, \delta}})\otimes \sL_{\C^*}|_{\Delta^*_\epsilon}\big)
&\cong H^{k+1}_c\big(\Exp^{-1}(\Delta^\circ_\epsilon), \Exp_\epsilon^{*} Rf'_!(\sF|_{\Pi_{\epsilon, \delta}})\otimes \underline{A}_{\Exp^{-1}(\Delta^\circ_\epsilon)}\big)\\
&\cong H^{k+1}_c\big(\Exp^{-1}(\Delta^\circ_\epsilon), \Exp_\epsilon^{*} Rf'_!(\sF|_{\Pi_{\epsilon, \delta}})\big).
\end{split}
\end{equation}
Since $\Exp^{-1}(\Delta^\circ_\epsilon)$ is a real two-dimensional contractible manifold, $H_c^2(\Exp^{-1}(\Delta^\circ_\epsilon), A)\cong A$ and $H_c^k(\Exp^{-1}(\Delta^\circ_\epsilon), A)=0$ for $k\neq 2$. Since $Rf'_!(\sF|_{\Pi_{\epsilon, \delta}})$ is locally constant, and its stalk is isomorphic to $R\Gamma_c(M_f, \sF|_{M_f})$, we get by the K\"unneth formula that
\begin{equation}\label{eq_A5}
H^{k+1}_c\big(\Exp^{-1}(\Delta^\circ_\epsilon), \Exp_\epsilon^{*} Rf'_!(\sF|_{\Pi_{\epsilon, \delta}})\big)\cong H^{k-1}_c(M_f, \sF|_{M_f}).
\end{equation}
Now, the desired formula \eqref{eq_A1} follows from equations \eqref{eq_A2}, \eqref{eq_A3}, \eqref{eq_A4} and \eqref{eq_A5}.
\end{proof}

\begin{proof}[Proof of Theorem \ref{thm_A}]
Since $\sL^f_{U}$ is a local system of free $A$-modules, the tensor product $\otimes_A \sL^f_{U}$ is t-exact. Since $j$ is an open embedding of a hypersurface complement, it is a quasi-finite Stein morphism, hence $Rj_*$ is also t-exact (see, e.g., 
\cite[Proposition 3.29, Example~3.67, Theorem 3.70]{MS}). Now, the right t-exactness of $\Psi_f$ follows from the right t-exactness of $i^{*}$. 

For the left t-exactness of $\Psi_f$, we need to check the costalk vanishing conditions on each stratum. By taking a normal slice (see, e.g., \cite[Page 427]{Schu} and compare also with the proof of \cite[Theorem 4.22]{MS}), we can always reduce to the case when the stratum is zero dimensional. More precisely, it suffices to show that $H_x^k(X_0, \Psi_f(\sP))=0$ for $k<0$, where $x$ is a point in $X_0$ and $\sP$ is a perverse sheaf on $X$.
Let $M_f$ be the a (general) local Milnor fiber of $f$ at $x$. Then $M_f$ is transversal to a stratification of $U$ on which $\sP$ is locally constant. By \cite[Corollary 10.2.6]{KS} or \cite[Proposition 2.27]{MS}, $\sP|_{M_f}[-1]$ is a perverse sheaf on $M_f$. Since $M_f$ is a Stein manifold, the assertion of the theorem follows from Proposition \ref{prop_milnor} and Artin's vanishing Theorem \ref{av}.
\end{proof}

Next, we show how to derive the $t$-exactness of the multivariate Sabbah specialization functor from the univariate case of Theorem \ref{thm_A}. 
\begin{cor}\label{cor_Sabbah}
Under the notations of Definition \ref{defn_Sabbah}, the multivariate Sabbah specialization functor $\Psi_F: D^b_{wc}(X, A)\to D^b_{wc}(D, R)$ is $t$-exact with respect to the perverse $t$-structures. 
\end{cor}
\begin{proof} Without loss of generality, we assume that $n\geq 2$. 
Let $f=f_1\cdots f_n$, and let $f_U: U\to \C^*$ be the restriction of $f$ to $U$. As before, denote $f_U^*(\sL_{\C^*})$ and $F_U^*(\sL_{(\C^*)^n})$ by $\sL_U^f$ and $\sL_U^F$, respectively. 

Using the natural isomorphisms $\pi_1((\C^*)^n)\cong \Z^n$ and $\pi_1(\C^*)\cong \Z$, the holomorpic map 
\[
\Pi: (\C^*)^n\to \C^*, \;\; (z_1, \ldots, z_n)\mapsto z_1\cdots z_n
\]
induces the homomorphism 
\[
\xi: \Z^n\to \Z, \;\; (a_1, \ldots, a_n)\mapsto a_1+\cdots+a_n
\]
on the fundamental groups. Then we have a natural isomorphism of rank one $A[\Z]$-local systems,
%\begin{equation}\label{eq_lf}
\[
\sL_U^f\cong \sL_U^F\otimes_{A[\Z^n]} A[\Z]
\]
%\end{equation}
where the $A[\Z^n]$-module structure on $A[\Z]$ is induced by $\xi$. 

Fix a splitting $\Z^n=\ker(\xi)\oplus \Z$ of the short exact sequence
\[
0\to \ker(\xi)\to \Z^n\xrightarrow[]{\xi} \Z\to 0,
\]
which induces a splitting of the short exact sequence of affine tori
\[
1\to \ker(\Pi)\to (\C^*)^n\xrightarrow[]{\Pi} \C^*\to 1.
\]
By the definition of the universal local system, the above splitting of affine tori induces an isomorphism of $A$-local systems
\[
\sL_{(\C^*)^n}\cong \sL_{\ker(\Pi)}\otimes_A \sL_{\C^*}.
\]
%Then we have a ring isomorphism $A[\Z^n]\cong A[\ker(\xi)]\otimes_A A[\Z]$. In particular, both $A[\ker(\xi)]$ and $A[\Z]$ can be considered as quotient rings of $A[\Z^n]$.
We denote the pullback $F^*\sL_{\ker(\Pi)}$ by $\sL_U'$. Then taking the pullback of the above isomorphism, we have
\[
\sL_U^F\cong \sL'_U\otimes_A \sL_U^f
\]
as $A$-local systems. 

Since $\sF$ is a weakly $A$-constructible complex on $X$ and $\sL'_U$ is a local system of free $A$-modules, $Rj_*(\sF|_U\otimes_A \sL'_{U})$ is a weakly $A$-constructible complex on $X$ (see \cite[Theorem~2.6(c)]{MS}). 
Therefore, considering $\Psi_F(\sF)$ as an object in $D^b_{wc}(D, A[\Z])$, we have
\[
\Psi_F(\sF)\cong i^*Rj_*\big(\sF|_U\otimes_A \sL^F_{U}\big)\cong i^*Rj_*\big(\big(\sF|_U\otimes_A \sL'_{U}\big)\otimes_A \sL^f_U\big)\cong \Psi_f\big(Rj_*\big(\sF|_U\otimes_A \sL'_{U}\big)\big).
%\Psi_f(\sF)\otimes_A A[\ker(\xi)].
\]
Since $\sF$ is a weakly $A$-perverse sheaf on $X$ and $\sL'_U$ is a local system of free $A$-modules, the tensor product $\sF|_U\otimes_A \sL'_{U}$ is a weakly $A$-perverse sheaf on $U$. Since $j: U\hookrightarrow X$ is an open embedding whose complement is a divisor, $j$ is a quasi-finite Stein mapping, and hence 
the pushforward $Rj_*\big(\sF|_U\otimes_A \sL'_{U}\big)$ is a weakly $A$-perverse sheaf on $X$. 
By Theorem~\ref{thm_A}, $\Psi_f(Rj_*(\sF|_U\otimes_A \sL'_{U}))$ is a weakly $A$-perverse sheaf. Since the definition of the perverse $t$-structure does not involve the ring of coefficients, we conclude that $\Psi_F(\sF)$ is also perverse as a (weakly) $R$-constructible complex. 
\end{proof}

\section{Local vanishing of the multivariate Sabbah specialization functor}\label{sec_local}
Let $X$ be a complex manifold. For $1\leq k\leq n$, let $f_k: X\to \C$ be holomorphic functions as in Section \ref{sec2.3}, with $D_k=f_k^{-1}(0)$  the corresponding divisors. Set $$F=(f_1, \ldots, f_n): X\to \C^n.$$
For any subset $I\subset\{1, \ldots, n\}$, let $$D_I=\bigcap_{k\in I}D_k \ \ \ {\rm and} \ \ \  D_I^\circ=D_I \setminus \bigcup_{m\notin I}D_m.$$ 
For a subset $J\subset \{1, \ldots, n\}$, we let $D_{>J}=\bigcup_{I\supsetneq J}D_I$. We also let $$D_{\geq m}=\bigcup_{|I|= m}D_I \ \ \ {\rm and} \ \ \  D_{\geq m}^\circ=D_{\geq m}\setminus D_{\geq m+1}.$$
Let $D=D_{\geq 1}$ and $U=X\setminus D$. Let $F_U: U\to (\C^*)^n$ be the restriction of $F$ to $U$.

If $S$ is an open submanifold of $X$, we denote the open embedding by $j_S: S\hookrightarrow X$. If $S$ is a locally closed, but not open, subvariety of $X$, we denote the inclusion map by $i_S: S\hookrightarrow X$.

Finally, denote as before by  $\sL_{(\C^*)^n}$  the universal $\K[\pi_1((\C^*)^n)]$-local system on $(\C^*)^n$, and let $\sL^F_U=F_U^*(\sL_{(\C^*)^n})$. 

\begin{remark}
Here we do not assume that the divisors $D_k$ define a locally complete intersection. So the codimension of $D_I^\circ$ is only $\leq |I|$. In fact, $D_I^\circ$ may not be equidimensional. 
\end{remark}

Given any weakly $A$-constructible complex $\sF$ on $X$ and any nonempty subset $I\subset \{1, \ldots, n\}$, we define 
\[
\Psi_{D_I^\circ}(\sF)\coloneqq i_{D_I^\circ}^* Rj_{U*}\big(\sF|_U \otimes_\K \sL^F_{U}\big),
\]
i.e., the restriction of the Sabbah specialization complex $\Psi_F(\sF)$ to $D_I^\circ$. 

In this section, we prove the following.
\begin{thm}\label{thm_local}
Let $R=\K[\pi_1((\C^*)^n)]$. 
Then the functor 
\[
\Psi_{D_I^\circ}: D^b_{wc}(X, \K)\to D^b_{wc}(D_I^\circ, R)
\]
is $t$-exact with respect to the perverse $t$-structures. 
\end{thm}
\begin{proof}
First we prove the assertion in the case when the divisors $D_1, \ldots, D_n$ are individually smooth and their union has normal crossing singularities. 

Under the normal crossing assumption, we prove the theorem using induction on $|I|$, the cardinality of $I$. When $|I|=1$, the assertion follows from Corollary \ref{cor_Sabbah} (since restriction to opens is t-exact). Fixing an integer $m\geq 2$, we assume that the assertion holds for all $I$ with $|I|<m$, and we want to show that the assertion holds for $I$ with $|I|=m$. 

\begin{lemma}\label{lemma_right}
The functor $\Psi_{D_I^\circ}$ is right $t$-exact, i.e., it maps $^p D^{\leq 0}_{wc}(X, \K)$ to $^p D^{\leq 0}_{wc}(D_I^\circ, R)$. 
\end{lemma}
\begin{proof}
By definition, for any weakly constructible complex $\sF$ on $X$, we have $\Psi_{D_I^\circ}(\sF)\cong i_{D_I^\circ, D}^* \Psi_{F}(\sF)$, where $i_{D_I^\circ, D}: D_I^\circ\hookrightarrow D$ is the inclusion map.
The right $t$-exactness of $\Psi_{D_I^\circ}$ follows from Corollary \ref{cor_Sabbah}, together with the fact that the pullback functor $i_{D_I^\circ, D}^*$ is right $t$-exact (here we write $i_{D_I^\circ, D}$ as a composition of the closed inclusion of $D_I$ into $D$, followed by the open inclusion of $D_I^\circ$ into $D_I$), see, e.g., \cite[Theorem 5.2.4 (iii)(iv)]{Di}.
\end{proof}

To show the left $t$-exactness of $\Psi_{D_I^\circ}$, we first prove it for zero-dimensional strata. 
%In other words, we first show the following. 
Here, we consider the following Proposition \ref{prop_point} as part of the proof of the theorem, and we will use the inductive hypothesis in the proof of Proposition \ref{prop_point}. 
\begin{prop}\label{prop_point}
Under the above notations and the normal crossing assumption, let $x\in D_I^\circ$ be an arbitrary point, and denote the closed embedding by $i_x: \{x\}\hookrightarrow D_I^\circ$. Then for any constructible complex $\sF$ in $^p D^{\geq 0}_{wc}(X, \K)$, we have
\begin{equation}\label{eq_local0}
H^k\big(i_x^! \Psi_{D_I^\circ}(\sF)\big)=0\quad \text{for any}\;\; k<0. 
\end{equation}
\end{prop}

To avoid introducing more notations for the inclusion maps from $x$ to various spaces, we use the  local cohomology notation instead of the cohomology of the exceptional pullback to $x$. The following lemma is a local cohomology version of \cite[spectral sequence~(10)]{Ar}. Nevertheless, we give here a different proof. 
\begin{lemma}\label{lemma_E1}
Assume $|I|=m$. There is a spectral sequence 
\begin{multline*}
E_1^{pq}=
\begin{cases}
H^{p+q+1}_x\Big(X, (i_{D_{\geq -q}^\circ})_! i_{D_{\geq -q}^\circ}^*Rj_{U*}\big(\sF|_U\otimes_\K \sL^F_{U}\big) \Big),&\text{when}\; 1-m\leq q\leq -1 \\
H^{p}_x\Big(X, i_{D*} i_{D}^*Rj_{U*}\big(\sF|_U\otimes_\K \sL^F_{U}\big) \Big),& \text{when}\; q=0\\
0, & \text{otherwise}
\end{cases}
\\
\Longrightarrow H^{p+q}_x\big(D_I^\circ, \Psi_{D_I^\circ}(\sF)\big).
\end{multline*}
\end{lemma}

\begin{proof}
Consider the double complex $\sA^{\bullet, \bullet}$ of weakly constructible sheaves on $X$ defined by
\[
\sA^{p, q}=(i_{D_{\geq p+1}})_* i_{D_{\geq p+1}}^*Rj_{U*}\big(\sF|_U\otimes_\K \sL^F_{U}) 
\]
when $0\leq p=-q\leq m-1$ and when $0\leq p=-1-q\leq m-2$. For other $p, q$, we let $\sA^{p, q}=0$. 

By base change, we have a natural isomorphism 
\[
(i_{D_{\geq p+1}})_* i_{D_{\geq p+1}}^*(i_{D_{\geq p}})_* i_{D_{\geq p}}^*Rj_{U*}\big(\sF|_U\otimes_\K \sL^F_{U})\cong (i_{D_{\geq p+1}})_* i_{D_{\geq p+1}}^*Rj_{U*}\big(\sF|_U\otimes_\K \sL^F_{U}),
\]
and hence the adjunction distinguished triangle can be written as
\begin{multline}\label{eq_adjunction}
%\begin{split}
(i_{D_{\geq p}^\circ})_! i_{D_{\geq p}^\circ}^*Rj_{U*}\big(\sF\otimes_\K \sL^F_{U}\big)\to (i_{D_{\geq p}})_* i_{D_{\geq p}}^*Rj_{U*}\big(\sF\otimes_\K \sL^F_{U}\big)\\
\to 
%(i_{D_{\geq p+1}})_* i_{D_{\geq p+1}}^*(i_{D_{\geq p}})_* i_{D_{\geq p}}^*Rj_{U*}\big(\sF\otimes_\K \sL_{U})\cong 
(i_{D_{\geq p+1}})_* i_{D_{\geq p+1}}^*Rj_{U*}\big(\sF\otimes_\K \sL^F_{U}\big)\xrightarrow{+1}
%\end{split}
\end{multline}
Now, we define all the horizontal differentials $d'$ to be zero, except for $1\leq p\leq m-1$, where we let $d': \sA^{p-1, -p}\to \sA^{p, -p}$ to be the second map in \eqref{eq_adjunction}. Similarly, we define all the vertical differentials $d''$ to be zero, except for $0\leq p\leq m-2$, when we let $d'': \sA^{p, -p-1}\to \sA^{p, -p}$ to be the identity maps. 

Since all column complexes $(\sA^{p, \bullet}, d'')$ are exact, except for $p=m-1$, and the $(m-1)$-th column is equal to $(i_{D_{\geq m}})_* i_{D_{\geq m}}^*Rj_{U*}\big(\sF|_U\otimes_\K \sL^F_{U}\big)[m-1]$, we have an isomorphism
\[
\textrm{tot}(\sA^{\bullet, \bullet})\cong (i_{D_{\geq m}})_* i_{D_{\geq m}}^*Rj_{U*}\big(\sF|_U\otimes_\K \sL^F_{U}\big)
\]
in $D^b_{w-c}(X, R)$, where $\textrm{tot}(\sA^{\bullet, \bullet})$ is the total complex of $\sA^{\bullet, \bullet}$ considered as an object in $D^b_{w-c}(X, R)$. 

Consider the filtration $F_q\coloneqq \sA^{\bullet, \leq q}$ of $\sA^{\bullet, \bullet}$ by row truncations. The graded pieces of the filtration are the rows in $\sA^{\bullet, \bullet}$. Using the adjunction distinguished triangle, we have
\begin{equation}\label{eq_gr}
\textrm{tot}\big(\textrm{Gr}_q(\sA^{\bullet, \bullet})\big)\cong 
\begin{cases}
(i_{D_{\geq -q}^\circ})_! i_{D_{\geq -q}^\circ}^*Rj_{U*}\big(\sF|_U\otimes_\K \sL^F_{U}\big)[1] & \text{when}\; 1-m\leq q\leq -1\\
i_{D*} i_{D}^*Rj_{U*}\big(\sF|_U\otimes_\K \sL^F_{U}\big)& \text{when}\; q=0\\
0&\text{otherwise},
\end{cases}
\end{equation}
where $\textrm{tot}(\textrm{Gr}_q(\sA^{\bullet, \bullet}))$ is the total complex of $\textrm{Gr}_q(\sA^{\bullet, \bullet})$ considered as an object in $D^b_{w-c}(X, R)$.
Since $x\in D_I^\circ$ and $D_I^\circ$ is open in $D_{\geq m}$, the complexes 
$R(i_{D_{\geq m}})_* i_{D_{\geq m}}^*Rj_{U*}\big(\sF|_U\otimes_\K \sL^F_{U}\big)$ and 
$R\big(i_{D_I^\circ}\big)_* i_{D_I^\circ}^* Rj_{U*}\big(\sF|_U\otimes_\K \sL^F_{U}\big)$ are quasi-isomorphic in a neighbourhood of $x$. Hence we have  isomorphisms
\begin{align}\label{eq_twoline}
\begin{split}
H^{p+q}_x\Big(X, (i_{D_{\geq m}})_* i_{D_{\geq m}}^*Rj_{U*}\big(\sF|_U\otimes_\K \sL^F_{U}\big)\Big)&\cong H^{p+q}_x\Big(X, R(i_{D_I^\circ})_* i_{D_I^\circ}^* Rj_{U*}\big(\sF|_U\otimes_\K \sL^F_{U}\big)\Big) \\
&\cong H^{p+q}_x\big(D_I^\circ, \Psi_{D_I^\circ}(\sF)\big).
\end{split}
\end{align}

Taking local cohomology of the filtered complex $\textrm{tot}(\sA^{\bullet, \bullet})$, we have a spectral sequence
\[
E^{pq}_1=H^{p+q}_x\big(X, \mathrm{tot}\big(\textrm{Gr}_q(\sA^{\bullet, \bullet})\big)\big)\Rightarrow H_x^{p+q}\big(X, \mathrm{tot}(\sA^{\bullet, \bullet})\big).
\]
The isomorphisms \eqref{eq_gr} and \eqref{eq_twoline} yield the spectral sequence in the lemma. 
\end{proof}

\begin{proof}[Proof of Proposition \ref{prop_point}] 
We identify $A[\pi_1((\C^*)^n)]$ with $A[t_1^\pm, \ldots, t_n^\pm]$ using the standard isomorphisms 
\[
A[\pi_1((\C^*)^n)]\cong A[\Z^n]\cong A[t_1^\pm, \ldots, t_n^\pm].
\]
Let $B_x$ be a small polydisc in $X$ centered at $x$, and let $U_x=B_x\cap U$. Without loss of generality, we assume that $I=\{1, \ldots, m\}$. Since $D$ is a normal crossing divisor, we have a natural isomorphism $A[\pi_1(U_x)]\cong A[t_1^{\pm}, \ldots, t_{m}^{\pm}]$. Let $\sL_{U_x}$ be the universal $\K[t_1^{\pm}, \ldots, t_{m}^{\pm}]$-local system on $U_x$. As $\K[t_1^{\pm}, \ldots, t_{m}^{\pm}]$-local systems on $U_x$, we have a non-canonical isomorphism
\[
\sL^F_U|_{U_x}\cong \sL_{U_x}\otimes_\K \K[t_{m+1}^\pm, \ldots, t_n^{\pm}].
\]
Therefore, 
\begin{align*}
H^k\big(i_x^! \Psi_{D_I^\circ}(\sF)\big)&=H^k\big(i_x^!i_{D_I^\circ}^* Rj_{U*}\big(\sF|_U\otimes_\K \sL^F_{U}\big)\big)\\
&\cong H^k_x\big(B_x, i_{D_I^\circ}^* Rj_{U*}\big(\sF|_U\otimes_\K \sL_{U_x}\big)\big)\otimes_\K \K[t_{m+1}^\pm, \ldots, t_n^{\pm}].
\end{align*}
Thus, it suffices to prove the vanishing \eqref{eq_local0} under the following assumption, which we will make for the rest of this proof. 
\begin{assumption}
The space $X=B_x=\Delta^l$ is a small polydisc in $\C^l$ centered at the origin $x$, $n=m$ and $f_1, \ldots, f_m$ are the first $m$ coordinate functions. In particular, $\sL^F_U=\sL_U$, and $U=U_x$.
\end{assumption}

We claim that
\begin{equation}\label{eq_vanishing}
H^{p+q+1}_x \Big(X, \big(i_{D_{\geq -q}^\circ}\big)_! i_{D_{\geq -q}^\circ}^*Rj_{U*}\big(\sF|_U\otimes_\K \sL_{U}\big) \Big)=0
\end{equation}
for $1-m\leq q\leq -1$ and $p+q<0$. In fact, let $J\subset \{1, \ldots, m\}$ with $|J|=-q$. To show \eqref{eq_vanishing}, it suffices to show that for $1-m\leq q\leq -1$ and $p+q<0$,
\begin{equation}\label{eq_J}
H^{p+q+1}_x \Big(X, (i_{D_J^\circ})_! i_{D_J^\circ}^*Rj_{U*}\big(\sF|_U\otimes_\K \sL_{U}\big) \Big)=0.
\end{equation}
Without loss of generality, we assume that $J=\{1, \ldots, -q\}$. By the above assumptions, $U=(\Delta^\circ)^m\times \Delta^{l-m}$, where $\Delta$ is a small disc in $\C$ centered at the origin, and $\Delta^\circ=\Delta\setminus \{0\}$. We further decompose $U$ as $U=(\Delta^\circ)^{-q}\times (\Delta^\circ)^{m+q}\times \Delta^{l-m}$, and we let $\sL_{U}^J$ and $\sL_{U}^{J^c}$ be the pullback of the universal local systems $\sL_{(\Delta^\circ)^{-q}}$ and $\sL_{(\Delta^\circ)^{m+q}\times \Delta^{l-m}}$ to $U$, respectively. Then as $A[t_1^\pm, \ldots, t_{m}^\pm]$-local systems,
\[
\sL_U\cong \sL_{U}^J\otimes_A \sL_{U}^{J^c}
\]
where the $A[t_1^\pm, \ldots, t_{m}^\pm]$-module structures on $\sL_{U}^J$ and $\sL_{U}^{J^c}$ are induced by the natural projections $\pi_1(U)\to \pi_1((\Delta^\circ)^{-q})$ and $\pi_1(U)\to \pi_1((\Delta^\circ)^{m+q})$, respectively. Thus, we have
\begin{align}\label{eq_J2}
\begin{split}
i_{D_J^\circ}^*Rj_{U*}\big(\sF|_U\otimes_\K \sL_{U}\big)&\cong i_{D_J^\circ}^*Rj_{U*}\big(\sF|_U\otimes_\K \sL_{U}^J\otimes_A \sL_{U}^{J^c}\big)\\
&\cong i_{D_J^\circ}^*Rj_{U*}\big(\sF|_U\otimes_\K \sL_{U}^J\big)\otimes_A \sL_{D_J^\circ}
\end{split}
\end{align}
where $\sL_{D_J^\circ}$ is the universal local system on $D^\circ_J$, and the second isomorphism follows from the fact that $\sL_{U}^{J^c}$ extends as an $A[t_{-q+1}^\pm, \ldots, t_m^\pm]$-local system to $X\setminus (D_{-q+1}\cup \cdots \cup D_m)$, and the restriction of the extension to $D^\circ_J$ is isomorphic to $\sL_{D_J^\circ}$. 

Applying the inductive hypothesis in Theorem \ref{thm_local} to the space $X\setminus (D_{-q+1}\cup \cdots \cup D_m)$ and functions $f_1, \ldots, f_{-q}$, it follows that 
\[
\sG\coloneqq i_{D_J^\circ}^*Rj_{U*}\big(\sF|_U\otimes_\K \sL_U^J\big) \in \, ^pD^{\geq 0}_{w-c}\big(D_J^\circ, A[t_1^\pm, \ldots, t_{-q}^\pm]\big).
\]
%Consider $\sG$ as a weakly $A$-constructible complex, 
By \eqref{eq_J2}, we have %the left-hand side of \eqref{eq_J} is isomorphic to
\begin{equation}\label{eq_iso1}
H^{p+q+1}_x \Big(X, (i_{D_J^\circ})_! i_{D_J^\circ}^*Rj_{U*}\big(\sF|_U\otimes_\K \sL_{U}\big) \Big)\cong H^{p+q+1}_x \big(X, (i_{D_J^\circ})_! \big(\sG \otimes_A \sL_{D_J^\circ}\big)\big).
\end{equation}
Consider the distinguished triangle
\[
(i_{D_J^\circ})_! \big(\sG \otimes_A \sL_{D_J^\circ}\big)\to R(i_{D_J^\circ})_* \big(\sG \otimes_A \sL_{D_J^\circ}\big)\to (i_{D_{>J}})_*i^*_{D_{>J}}R(i_{D_J^\circ})_* \big(\sG \otimes_A \sL_{D_J^\circ}\big)\xrightarrow{+1}.
\]
Since $i_x^!R(i_{D_J^\circ})_* \big(\sG \otimes_A \sL_{D_J^\circ}\big)=0$, the local cohomology long exact sequence implies that 
\begin{align}\label{eq_iso2}
\begin{split}
H^{p+q+1}_x\big(X, (i_{D_J^\circ})_! \big(\sG \otimes_A \sL_{D_J^\circ}\big)\big)
&\cong H^{p+q}_x\big(X, (i_{D_{>J}})_*i^*_{D_{>J}}R(i_{D_J^\circ})_* \big(\sG \otimes_A \sL_{D_J^\circ}\big)\big)\\
&\cong H^{p+q}_x\big(D_{>J}, i^*_{D_{>J}}R(i_{D_J^\circ})_* \big(\sG \otimes_A \sL_{D_J^\circ}\big)\big).
\end{split}
\end{align}
Notice that the last term of the above isomorphism is equal to the $(p+q)$-th cohomology of the costalk at $x$ of the multivariate Sabbah specialization functor applied to $\sG$ with respect to the holomorphic functions $f_i|_{D_J}$ on $D_J$ for $i\in \{1, \ldots, m\}\setminus J$. Thus, by Corollary~\ref{cor_Sabbah}, 
\begin{equation}\label{eq_vanish0}
H^{p+q}_x\big(X, (i_{D_{>J}})_*i^*_{D_{>J}}R(i_{D_J^\circ})_* \big(\sG \otimes_A \sL_{D_J^\circ}\big)\big)=0
\end{equation}
when $p+q<0$. Combining equations \eqref{eq_iso1}, \eqref{eq_iso2} and \eqref{eq_vanish0}, we have
\begin{equation}\label{eq_vanish1}
H^{p+q+1}_x\Big(X, (i_{D_J^\circ})_! i_{D_J^\circ}^*Rj_{U*}\big(\sF|_U\otimes_\K \sL_{U}\big) \Big)=0
\end{equation}
when $p+q<0$. 

Notice that
\[
H^{p+q}_x\big(X, i_{D*} i_{D}^*Rj_{U*}\big(\sF|_U\otimes_\K \sL_{U}\big)\big)\cong H^{p+q}_x\big(D, i_{D}^*Rj_{U*}\big(\sF|_U\otimes_\K \sL_{U}\big)\big)
\]
and the right-hand side is equal to the $(p+q)$-th cohomology of the costalk at $x$ of the Sabbah specialization $\Psi_F(\sF)$. Thus, by Corollary~\ref{cor_Sabbah}, we have
\begin{equation}\label{eq_vanish2}
H^{p+q}_x\big(X, i_{D*} i_{D}^*Rj_{U*}\big(\sF|_U\otimes_\K \sL_{U}\big)\big)=0
\end{equation}
when $p+q<0$. 

Therefore, the vanishing \eqref{eq_local0} follows from Lemma \ref{lemma_E1} and equations \eqref{eq_vanish1}, \eqref{eq_vanish2}. 
\end{proof}

We are now ready to finish the proof of Theorem \ref{thm_local} under the assumption that $D$ is a normal crossing divisor. In Lemma \ref{thm_local}, we have proved the right $t$-exactness of $\Psi_{D_I^\circ}$. By Proposition \ref{prop_point}, we know the left $t$-exactness of the functor $\Psi_{D_I^\circ}$ at zero-dimensional strata. The proof of the left $t$-exactness at higher dimensional strata can be reduced to the case of zero-dimensional strata by using normal slices (see, e.g., \cite[Page 427]{Schu}). This finishes the proof of the theorem under the assumption that $D$ is a normal crossing divisor. 

In general, the assertion can be reduced to the simple normal crossing case by considering the multivariate graph embedding 
\[
F^\dagger: X\to X\times \C^n, \;\; x\mapsto (x, F(x)),
\]
which restricts to a closed embedding $U\to X\times (\C^*)^n$. More precisely, the local vanishing \eqref{eq_local0} of Sabbah's specialization functor for $F: X\to \C^n$ can be reduced to the local vanishing of the Sabbah specialization functor for the projection $p_2:X\times \C^n\to \C^n$. 
This is due to the following natural isomorphism 
\[
R{\hat F}^\dagger_*\Psi_F(\sF) \cong \Psi_{p_2}(RF^\dagger_*\sF),
\]
where ${\hat F}^\dagger$ is the restriction of $F^\dagger$ to $D$.
\end{proof}

The next corollary shows that Theorem \ref{thm_local} can be considered as a generalization of the $t$-exactness of the Sabbah specialization functor. 
\begin{cor}
Let $R=\K[\pi_1((\C^*)^n)]$, and define the functor $\Psi_{D_I}$ by
\[
\Psi_{D_I}: D^b_{wc}(X, \K)\to D^b_{wc}(D_I, R), \quad\sF\mapsto i_{D_I}^* Rj_{U*}\big(\sF|_U\otimes_\K \sL_{U}\big).
\]
Then $\Psi_{D_I}$ is $t$-exact with respect to the perverse $t$-structures. 
\end{cor}
\begin{proof}
It suffices to show that if $\sP$ is a weakly constructible $A$-perverse sheaf on $X$, then $\Psi_{D_I}(\sP)$ is perverse. Notice that $\Psi_{D_I}(\sP)$ is equal to an iterated extension of constructible complexes $(i_{D_J, D_I})_*(i_{D_J^\circ, D_J})_!\Psi_{D_J^\circ}(\sP)$  for $J\supset I$. Since $D_J^\circ$ is a hypersurface complement in $D_J$, and $D_J$ is closed in $D_I$, Theorem \ref{thm_local}  implies that  $(i_{D_J, D_I})_*(i_{D_J^\circ, D_J})_!\Psi_{D_J^\circ}(\sP)$ is a perverse sheaf supported on $D_J$. Since extensions of perverse sheaves are also perverse, $\Psi_{D_I}(\sP)$ is a perverse sheaf on $D_I$. 
\end{proof}

\section{Non-abelian Mellin transformations}
First, we recall and extend the notations of Theorem \ref{thm_main}. 

Let $X$ be a compact complex manifold. Let $E=\bigcup_{1\leq k\leq d}E_k$ be a normal crossing divisor on $X$, and let $U=X\setminus E$ with inclusion map $j\colon U \hookrightarrow X$. 
For any nonempty subset $I\subset\{1, \ldots, d\}$, let $E_I=\bigcap_{i\in I}E_i$ and $E_I^\circ=E_I \setminus \bigcup_{j\notin I}E_j$. 
Let $E_{\geq m}=\bigcup_{|I|= m}E_I$, and let $E_{\geq m}^\circ=E_{\geq m}\setminus E_{\geq m+1}$.
For any open submanifold $S$ of $X$, we denote the open embedding by $j_S\colon S\hookrightarrow X$. For a locally closed, but not open, submanifold $S$ of $X$, we denote the inclusion map by $i_S\colon S\hookrightarrow X$. 
Let $\sL_U$ be the universal $\K[\pi_1(U)]$-local system on $U$. 

The following is simply a reformulation of Lemma \ref{lemma_YZ}, adapted to the above notations. 

\begin{lemma}\label{cor_sum}
Let $x$ be a point in $E$, let $B_x$ be a small ball in $X$ centered at $x$, and let $U_x=B_x\cap U$. If the map $\pi_1(U_x)\to \pi_1(U)$ induced by the inclusion is injective, then $\sL_U|_{U_x}$ is isomorphic to a direct sum of possibly infinitely many copies of $\sL_{U_x}$, where $\sL_{U_x}$ is the universal $\K[\pi_1(U_x)]$-local system on $U_x$. 
\end{lemma}

\begin{cor}\label{cor_local2}
Under the above notations, let $\sP\in D^b_c(X, \K)$ be an $\K$-perverse sheaf. Given a nonempty subset $I\subset \{1, \ldots, d\}$, assume that at every point $x\in E_I^\circ$, the local fundamental group maps injectively to the global fundamental group, that is, the condition (2) in Theorem~\ref{thm_main} holds. Then, 
\[
i^*_{E_I^\circ} Rj_*(\sP|_U\otimes \sL_{U})
\]
is a weakly constructible $A$-perverse sheaf on $E_I^\circ$. 
\end{cor}
\begin{proof}
It suffices to check the statement locally on $E_I^\circ$. For an arbitrary point $x\in E_I^\circ$, let $B_x$ be a small ball in $X$ centered at $x$ and let $U_x=B_x\cap U$. By Lemma \ref{cor_sum}, $\sL_U|_{U_x}$ is a direct sum of possibly infinitely many copies of $\sL_{U_x}$. Let $i_{E_I^\circ\cap B_x, B_x}\colon E_I^\circ\cap B_x\to B_x$ and $j_{U_x, B_x}\colon U_x\to B_x$ be the closed and open embeddings, respectively. 
The restriction of the complex $i^*_{E_I^\circ} Rj_*(\sP|_U\otimes \sL_{U})$ to $E_I^\circ\cap B_x$ is equal to $(i_{E_I^\circ\cap B_x, B_x})^* R(j_{U_x, B_x})_*(\sP|_{U_x}\otimes \sL_{U}|_{U_x})$, which is just a direct sum of copies of $(i_{E_I^\circ\cap B_x, B_x})^* R(j_{U_x, B_x})_*(\sP|_{U_x}\otimes \sL_{U_x})$. Hence, by Theorem~\ref{thm_local}, the weakly constructible complex $(i_{E_I^\circ\cap B_x, B_x})^* R(j_{U_x, B_x})_*(\sP|_{U_x}\otimes \sL_{U}|_{U_x})$ is an $A$-perverse sheaf on $E_I^\circ\cap B_x$.   
\end{proof}

\begin{proof}[Proof of Theorem \ref{thm_main}] 
Given any $A$-perverse sheaf $\sP$ on $X$, we need to show that
\begin{center}$H^k(\sM_*(\sP|_U))=0$ for $k\neq 0$.\end{center}
By assumption, $U$ is a Stein manifold. 
Since $\sP|_U\otimes_A \sL_U$ is a weakly constructible $A$-perverse sheaf, by Artin's vanishing Theorem \ref{av} we get:
\[
H^k(\sM_*(\sP))\cong H^k\big(U, \sP|_U\otimes_A \sL_U\big)=0 \quad \textrm{for}\ \ k>0.
\]
To show the vanishing in negative degrees, we consider the following distinguished triangle
\[
j_!\big(\sP|_U\otimes_\K \sL_U\big)\to Rj_*\big(\sP|_U\otimes_\K \sL_U\big) \to i_*i^*Rj_*\big(\sP|_U\otimes_A \sL_U\big)\xrightarrow{+1}
\]
where $i\colon E\hookrightarrow X$ and $j\colon U\hookrightarrow X$ are the closed and open embeddings, respectively. Since $X$ and $E$ are compact, the associated hypercohomology long exact sequence reads as:
\begin{equation}\label{eq_long}
\cdots\to H^k_c\big(U, \sP|_U\otimes_\K \sL_U\big)\to H^k\big(U, \sP|_U\otimes_\K \sL_U\big)\to H^k\big(E, i^*Rj_*\big(\sP|_U\otimes_A \sL_U\big)\big)\to %H^{k+1}_c(U, \sP\otimes_\K \sL_U)\to 
\cdots
\end{equation}

The following lemma is analogous to Lemma \ref{lemma_E1}. 
\begin{lemma}\label{lemma_ss2}
There exists a spectral sequence
\begin{multline*}
E_1^{pq}=
\begin{cases}H^{p+q}\Big(X, \big(i_{E_{\geq -q+1}^\circ}\big)_! i_{E_{\geq -q+1}^\circ}^*Rj_{U*}\big(\sP|_U\otimes_\K \sL_{U}\big) \Big) &\text{when}\;q\leq 0\\
0 &\text{when}\; q>0
\end{cases}
\\
\Longrightarrow H^{p+q}\big(X, i_*i^*Rj_*\big(\sP|_U\otimes_A \sL_U\big)\big).
\end{multline*}
Here, if $E_{\geq -q+1}^\circ=\emptyset$, our convention is that both $\big(i_{E_{\geq -q+1}^\circ}\big)_!$ and $i_{E_{\geq -q+1}^\circ}^*$ are zero functors. 
\end{lemma}
\begin{proof}
As in the proof of Lemma \ref{lemma_E1}, we define a double complex $\sB^{\bullet, \bullet}$ by
\[
\sB^{p, q}=(i_{E_{\geq p}})_* i_{E_{\geq p}}^*Rj_{U*}\big(\sP|_U\otimes_\K \sL_{U}\big)
\]
when $p=-q\geq 0$ or $p-1=-q\geq 0$. For other values of $p, q$, we let $\sB^{p, q}=0$. Consider the adjunction distinguished triangle
\begin{multline}\label{eq_adjunction2}
%\begin{split}
\big(i_{E_{\geq p}^\circ}\big)_! i_{E_{\geq p}^\circ}^*Rj_{U*}\big(\sP|_U\otimes_\K \sL_{U}\big)\to \big(i_{E_{\geq p}}\big)_* i_{E_{\geq p}}^*Rj_{U*}\big(\sP|_U\otimes_\K \sL_{U}\big)\\
\to 
\big(i_{E_{\geq p+1}}\big)_* i_{E_{\geq p+1}}^*Rj_{U*}\big(\sP|_U\otimes_\K \sL_{U}\big)\xrightarrow{+1}
%\end{split}
\end{multline}
We define all horizontal differentials $d'$ to be zero, except for $p\geq 0$, in which case we let $d': \sB^{p, -p}\to \sB^{p+1, -p}$ to be the second map in \eqref{eq_adjunction2}. We define all vertical differential $d''$ to be zero, except for $p\geq 1$, when we let $d'': \sB^{p, -p}\to \sB^{p, -p+1}$ to be the identity map. For the rest of the proof, we can use the same arguments as in the proof of Lemma~\ref{lemma_E1} with local cohomology replaced by hypercohomology. 
\end{proof}

Let us now get back to the proof of Theorem \ref{thm_main}. By Corollary \ref{cor_local2}, as a weakly constructible sheaf on $E_{\geq -q+1}^\circ$, the complex $i_{E_{\geq -q+1}^\circ}^*Rj_{U*}\big(\sP|_U\otimes_\K \sL_{U}\big)$ is perverse. By assumption, $E_{\geq -q+1}^\circ$ is a disjoint union of Stein manifolds. Thus, by Artin's vanishing Theorem \ref{av}, we have
\[
H^{p+q}\Big(X, \big(i_{E_{\geq -q+1}^\circ}\big)_! i_{E_{\geq -q+1}^\circ}^*Rj_{U*}\big(\sP|_U\otimes_\K \sL_{U}\big) \Big)\cong H_c^{p+q}\Big(E_{\geq -q+1}^\circ,  i_{E_{\geq -q+1}^\circ}^*Rj_{U*}\big(\sP|_U\otimes_\K \sL_{U}\big) \Big)=0
\]
for $q\leq 0$ and $p+q<0$. Therefore, by Lemma \ref{lemma_ss2}, we have
\begin{equation}\label{ea}
H^{k}\big(X, i_*i^*Rj_*\big(\sP|_U\otimes_A \sL_U\big)\big)=0,\quad \text{for}\; k<0.
\end{equation}
Furthermore, since $\sP|_U\otimes_A \sL_U$ is a weakly constructible $A$-perverse sheaf on the Stein manifold $U$, we get by Artin's vanishing Theorem \ref{av} that
\begin{equation}\label{eb}
H^k_c\big(U, \sP|_U\otimes_\K \sL_U\big)=0, \quad\textrm{for}\; k<0.
\end{equation}

By plugging \eqref{ea} and \eqref{eb} into the long exact sequence \eqref{eq_long}, we conclude that 
\[
H^k(\sM_*(\sP))\cong H^k\big(U, \sP|_U\otimes_A \sL_U\big)=0, \quad\text{for}\; k<0,
\]
thus completing the proof of Theorem \ref{thm_main}.
\end{proof}

\section{Some applications}

One of our motivations for studying the $t$-exactness of the non-abelian Mellin transformation is to extend results of Denham-Suciu \cite{DS} concerning duality spaces (in the sense of Bieri and Eckmann \cite{BE}). 

Let us first recall the following definition.

\begin{defn}
Let $U$ be a topological space with fundamental group $G$, which is homotopy equivalent to a connected, finite-type CW-complex. Let $\sL_U$ be the universal $\Z[G]$-local system on $U$. We say that $U$ is a {\it duality space of dimension $n$} if $H^k(U, \sL_U)=0$ for $k\neq n$, and $H^n(U, \sL_U)$ is a torsion-free $\Z$-module. 
\end{defn}

It is proved in \cite{DS} that for any $U$ satisfying the conditions of Theorem \ref{thm_main}, the topological space $U$ is a duality space of dimension $\dim_\C U$. 
In particular, complements of essential hyperplane arrangements, elliptic arrangements or toric arrangements are examples of duality spaces. The aim of this section is to construct new examples  of duality spaces that are non-affine or singular varieties (see Proposition \ref{prop_semismall} and Corollary \ref{cor_sing}).

\begin{prop}\label{prop_semismall}
Assume $U\subset X$ are complex algebraic varieties satisfying the conditions of Theorem \ref{thm_main}. Let $f\colon Y \to U$ be a proper birational semi-small map from a smooth complex algebraic variety $Y$. Then $Y$ is a duality space of dimension $\dim_\C Y$. In particular, blowing up $U$ along any codimension-two smooth subvariety gives rise to a duality space.
\end{prop}
\begin{proof}
We follow similar arguments as in the proof of \cite[Theorem 4.11(1)]{LMWb}. Let us first note that
$$H^k(Y,\sL_Y) \cong H^k(\sM_*^Y(\Z_Y)).$$
Let $d=\dim_\C Y$, and let $\bF=\bF_p$ be a field of positive prime characteristic $p$. Since there exists a birational and proper map between $U$ and $Y$, their fundamental groups are isomorphic (see, e.g., \cite[Page 494]{GH}). Therefore, $\sL_Y\cong f^*\sL_U$ with coefficients in $A=\Z, \Q$ or $\bF$. By projection formulas, we have
$$\sM_*^Y(\Z_Y) \cong \sM^U_*(Rf_*\Z_Y),$$
as well as
\begin{equation}\label{eq_m1}
\sM_*^Y(\Z_Y)\otimes_\Z \Q\cong \sM_*^Y(\Q_Y)\cong \sM^U_*(Rf_*\Q_Y)
\end{equation}
and
\begin{equation}\label{eq_m2}
\sM_*^Y(\Z_Y)\stackrel{L}{\otimes}_\Z \bF\cong \sM_*^Y(\bF_Y)\cong \sM^U_*(Rf_*\bF_Y).
\end{equation}
%where $\F=\F_p$ is a field of positive prime characteristic $p$. 
Since $f$ is semi-small, by \cite[Example~6.0.9]{Schu} we have the following,
\begin{enumerate}
\item $Rf_*(\Q_Y[d])$ and $Rf_*(\bF_Y[d])$ are perverse sheaves on $U$;
\item $Rf_*(\Z_Y[d])\in \,^p D^{\leq 0}_c(U, \Z)$ is semi-perverse. 
\end{enumerate}
By Theorem \ref{thm_main} and the isomorphisms \eqref{eq_m1} and \eqref{eq_m2}, the complexes $\sM_*^Y(\Z_Y[d])\otimes_\Z \Q$ and $\sM_*^Y(\Z_Y[d])\stackrel{L}{\otimes}_\Z \bF$ are concentrated in degrees zero, and the cohomology of $\sM_*^Y(\Z_Y[d])$ vanishes in positive degrees. Thus, by Lemma \ref{l63} below, the complex $\sM_*^Y(\Z_Y[d])$ is also concentrated in degree zero, and its cohomology in degree zero is a torsion-free $\Z$-module. In other words, $Y$ is a duality space of dimension $d$.
\end{proof}

\begin{lemma}\label{l63}
Let $N^\bullet$ be a bounded complex of free $\Z$-modules. Suppose that 
\begin{enumerate}
\item $H^k(N^\bullet\otimes_\Z \Q)=H^k(N^\bullet\otimes_\Z \bF_p)=0$ for any $k\neq 0$ and for any prime number $p$;
\item $H^k(N^\bullet)=0$ for $k>0$. 
\end{enumerate}
Then $H^k(N^\bullet)=0$ for $k\neq 0$, and $H^0(N^\bullet)$ is a torsion-free $\Z$-module. 
\end{lemma}
\begin{proof}
By (2), it suffices to show that $H^k(N^\bullet)=0$ for $k<0$, and $H^0(N^\bullet)$ is torsion-free. For any $k<0$, since $H^k(N^\bullet\otimes_\Z \Q)=0$, $H^k(N^\bullet)$ is a torsion $\Z$-module. Thus, it suffices to show that $H^k(N^\bullet)$ is torsion free for all $k\leq 0$. 

Suppose that for some $k\leq 0$, $H^k(N^\bullet)$ has nonzero torsion elements. Let $k_0$ be the smallest such $k$, and let $p$ be a prime number such that $H^{k_0}(N^\bullet)$ has nonzero $p$-torsion elements. Here, notice that if $\eta\in H^{k_0}(N^\bullet)$ has order $m>0$, and if $p$ is a prime divisor of $m$, then $\frac{m}{p}\eta$ is a $p$-torsion element. The $p$-torsion element in $H^{k_0}(N^\bullet)$ induces a short exact sequence,
\[
0\to \F_p\to H^{k_0}(N^\bullet)\to H^{k_0}(N^\bullet)/\F_p\to 0.
\]
Then as part of the associated long exact sequence, we have
\[
0= \mathrm{Tor}_2^\Z\big(H^{k_0}(N^\bullet)/\F_p, \F_p\big) \to \mathrm{Tor}_1^\Z(\F_p, \F_p)=\F_p \to \mathrm{Tor}_1^\Z(H^{k_0}(N^\bullet), \F_p),
\]
which implies that $\mathrm{Tor}_1^\Z(H^{k_0}(N^\bullet), \F_p)\neq 0$. 
By the universal coefficient theorem, there is a non-canonical isomorphism
\[
H^{k_0-1}(N^\bullet\otimes_\Z \F_p)\cong H^{k_0-1}(N^\bullet)\otimes_\Z \F_p \oplus \mathrm{Tor}_1^\Z(H^{k_0}(N^\bullet), \F_p).
\]
Thus, $H^{k_0-1}(N^\bullet\otimes_\Z \F_p)\neq 0$. Since $k_0\leq 0$, this contradicts our assumption~(1).
\end{proof}

\begin{prop}
Let $U$ be a complex manifold with a compactification $X$ satisfying the conditions in Theorem \ref{thm_main}, and let $Z\subset U$ be a connected closed analytic subvariety, which is also locally closed in $X$. Assume that $Z$ is a locally complete intersection, and the inclusion $Z\hookrightarrow U$ induces an isomorphism on the fundamental groups. Then $Z$ is a duality space of dimension $\dim_\C Z$. In particular, $U$ itself is a duality space of dimension $\dim_\C U$. 
\end{prop}

\begin{proof}
First of all, since $Z$ is locally closed in $X$, $R(j\circ i)_*\,\Z_Z\cong Rj_*\, j^{*}\, \bar{i}_*\,\Z_{\bar{Z}}$ is constructible on $X$, where $\bar{Z}$ is the closure of $Z$ in $X$, and $i: Z\to U$, $\bar{i}: \bar{Z}\to X$ and $j: U\to X$ are the inclusion maps (see, e.g., \cite[Theorem 2.5]{MS}). 

We use similar arguments and notations as in the proof of Proposition \ref{prop_semismall}. Let $d=\dim_\C Z$. 
Since $i\colon Z\to U$ induces an isomorphism between the fundamental groups, there is a natural isomorphism $\sM_*^Z(\Z_Z[d]) \cong\sM^U_*(i_*\,\Z_Z[d])$. Hence, it suffices to show that $\sM^U_*(i_*\,\Z_Z[d])$ is concentrated in degree zero, and its cohomology in degree zero is a torsion-free $\Z$-module. 

Since $Z$ is a locally complete intersection,  $\Z_Z[d]$, $\Q_Z[d]$ and $\bF_Z[d]$ are perverse sheaves on $Z$. Since $i$ is a closed embedding, $i_*\, \Z_Z[d]$, $i_*\, \Q_Z[d]$ and $i_*\, \bF_Z[d]$ are perverse sheaves on $U$. By Theorem~\ref{thm_main}, each of the complexes $\sM^U_*(i_*\, \Z_Z[d])$, $\sM^U_*(i_*\, \Q_Z[d])$ and $\sM^U_*(i_*\, \bF_Z[d])$ is concentrated in degree zero. As in the proof of Proposition \ref{prop_semismall}, we  conclude that $\sM^U_*(i_*\, \Z_Z[d])$ has torsion-free cohomology in degree zero. 
\end{proof}

\begin{cor}\label{cor_sing}
Let $Y\subset \P^n$ be a hypersurface such that the singular locus $Y_{\mathrm{sing}}$ has codimension at least 3. Let $D_1, \ldots, D_m\subset \P^n$ be a family of smooth hypersurfaces in general position and transversal to $Y$ such that $Y\cap D_1\cap \cdots \cap D_m=\emptyset$. Then, $Y\setminus (D_1\cup \cdots \cup D_m)$ is a duality space. 
\end{cor}
\begin{proof}
Set $U=\P^n\setminus (D_1\cup \cdots \cup D_m)$. It is clear that $U$ with compactification $\P^n$ satisfies the conditions in Theorem \ref{thm_main}. Note that  $Y\setminus (D_1\cup \cdots \cup D_m)$ is a hypersurface in $U$. By the above proposition, we only need to show that the inclusion map $Y\setminus (D_1\cup \cdots \cup D_m) \to U$  induces an isomorphism on the fundamental groups. By the Lefschetz hyperplane section theorem,
after intersecting with a generic projective linear space $L \subset \P^n$ with $\dim L=3$, we can assume that $Y$ is a smooth hypersurface $Y$ in $\P^3$ intersecting $D_1\cup \cdots \cup D_m$ transversally. 
Using the Lefschetz hyperplane section theorem, we get that
$$Y\setminus (D_1\cup \cdots \cup D_m)\to \P^3 \setminus (D_1\cup \cdots \cup D_m)$$
induces an isomorphism on the fundamental groups.
\end{proof}

We regard a complex manifold $U$ as in Theorem \ref{thm_main} as the affine counterpart of a complex projective aspherical manifold. The Mellin transformations of certain projective aspherical manifolds are discussed in \cite{LMWd}, where, under certain assumptions, we show that the Mellin transformation of a nontrivial constructible complex is nonzero (see \cite[Proposition 3.3 and Proposition 5.6]{LMWd}). 
We conjecture that this fact remains true in the more general setting of our Theorem \ref{thm_main}. 
\begin{conj}\label{conj_1}
Assume that $U\subset X$ are complex manifolds satisfying the conditions in Theorem \ref{thm_main}. Let $\sF \in D^b_c(X, \K)$ be a constructible complex such that $\sF|_U\neq 0$. Then $\sM_*(\sF|_U)\neq 0$.
\end{conj}

\begin{remark}
If $U\subset X$ are algebraic varieties, and if $U$ admits a quasi-finite map to some semi-abelian variety, then the conjecture holds. This can be proved by combining \cite[Proposition 5.6]{LMWd} and \cite[Proposition 5.4]{LMWc}. Such examples include complements of essential linear hyperplane arrangements, complements of toric arrangements and complements of elliptic arrangements. 
\end{remark}

%\textcolor{magenta}{The above conjecture holds if $U$ is very affine (closed subvariety of some affine torus), in particular, for hyperplane arrangements. This can be reduced to the result of Gabber-Loeser. }

A particular consequence of the above conjecture is that, when $U$ is an algebraic variety, the perverse $t$-structure on $D^b_c(U, \K)$ can be completely detected by the Mellin transformation. 
\begin{prop}
Assume that $U\subset X$ are complex manifolds satisfying the conditions in Theorem \ref{thm_main}, and assume that the above conjecture holds. Then for an $A$-constructible complex $\sF$ on $X$, $\sF|_U$ is perverse if and only if $\sM_*(\sF|_U)$ is concentrated in degree zero. 
\end{prop}
\begin{proof}
The ``only if" part is exactly Theorem \ref{thm_main}. To show the converse, suppose that $\sF|_U$ is not a perverse sheaf. Then there exists $k\neq 0$ such that $^p\sH^k(\sF|_U)\neq 0$. It follows from Theorem \ref{thm_main} that $H^k(\sM_*(\sF|_U))\cong H^0(\sM_*(\, ^p\sH^k(\sF|_U)))$, which is nonzero by Conjecture \ref{conj_1}. This is a contradiction to the assumption that $\sM_*(\sF|_U)$ is concentrated in degree zero. 
\end{proof}

%%%%%%%%%%%%%%%%%%%%%%%%%%%%%%%%%%

\end{document}